\documentclass[letterpaper,12pt]{amsart}

\usepackage{color}
\usepackage{amsmath,amsthm}
\usepackage{enumerate}
\usepackage{changes}
\usepackage{ulem,ifthen,xcolor,xkeyval,pdfcolmk}

\newcommand{\vol}{{\rm vol}\,}
\newcommand{\pa}{\partial}

\usepackage[abbrev]{amsrefs}
\usepackage{amssymb,amsmath,amsthm}

\newtheorem{thm}{Theorem}[section]

\newtheorem{prop}{Proposition}[section]

\newtheorem{Def}[thm]{Definition}
\newtheorem{lem}[thm]{Lemma}

\newtheorem{conj}[thm]{Conjecture}

\newtheorem{remark}[thm]{Remark}
\newtheorem{corT}{Corollary}[section]

\newcommand{\Nat}{\mathbb{N}}

\newcommand{\eps}{\varepsilon}
\newcommand{\Hilbert}{\mathcal{H}}
\newcommand{\Dom}{\mathfrak{D}}

\def \<{\langle}
\def \>{\rangle}

\def \R{\mathbb R}

\def \H{{\cal H}}

\def \H^0{{\cal H}^0 or}

\def \p{\partial}
\def \n{\nabla}
\def \beq{\begin{equation}}
\def \eeq{\end{equation}}

\def \n{\nabla}

\def \eref{\eqref}

\begin{document}

\title{On the spectrum of the Laplacian}

\author{Nelia Charalambous}
\address{Department of Mathematics and Statistics, University of Cyprus, Nicosia, 1678, Cyprus, and
Deparment of Mathematics, Instituto Tecnol\'ogico Aut\'onomo de M\'exico, Mexico D.F. 01000, Mexico} \email[Nelia Charalambous]{nelia@ucy.ac.cy}

 \author{Zhiqin Lu} \address{Department of
Mathematics, University of California,
Irvine, Irvine, CA 92697, USA} \email[Zhiqin Lu]{zlu@uci.edu}

\thanks{
The first author was partially supported by CONACYT of Mexico and is thankful to the Asociaci\'on Mexicana de Cultura A.C. The second author is partially supported by the DMS-12-06748.}
 \date{September 1, 2012}



\begin{abstract} In this article we prove a generalization of Weyl's criterion for the essential spectrum of a self-adjoint operator on a Hilbert space. We then apply this criterion to the Laplacian on functions over open manifolds and  get new results for its essential spectrum.
\end{abstract}

\maketitle

\section{Introduction}
Let $M$ be a complete noncompact Riemannian manifold  of dimension $n$ and denote by $\Delta$ the Laplacian acting on $\mathcal C_0^\infty(M)$. It is well known that the self-adjoint extension of $\Delta$ on $L^2(M)$ exists and is a unique nonpositive definite and densely defined linear operator.  We will also use $\Delta$ to denote this extension for the remaining paper.

The spectrum of $-\Delta$,  $\sigma(-\Delta)$,  consists of all points $\lambda\in \mathbb{C}$ for which $\Delta+\lambda I$ fails to be invertible. Since $-\Delta$ is nonnegative definite, its  $L^2$-spectrum  is contained in $[0,\infty)$. The  essential spectrum of  $-\Delta$, $\sigma_\textup{ess}(-\Delta)$, consists of the cluster points in the spectrum and of isolated eigenvalues of  infinite multiplicity.  The following result is due to Donnelly~\cite{Don2}: if there exists an infinite dimensional subspace $G$ in the domain of $\Delta$  such that
\begin{equation}\label{sdf}
\|\Delta u+\lambda u\|_{L^2}\leq\sigma\|u\|_{L^2}
\end{equation}
for all $u\in G$, then
\[
\sigma_\mathrm{ess}(-\Delta)\cap(\lambda-\sigma,\lambda+\sigma)\neq\emptyset.
\]
The functions $u$ are referred to as the approximate eigenfunctions corresponding to the eigenvalue $\lambda$. The above criterion is simple to apply and has directed the study of the essential spectrum of the Laplacian  for the last three decades. A related result of the above is as follows: let $u$ be a nonzero smooth function with compact support. If
~\eqref{sdf} is satisfied, then
\[
\sigma(-\Delta)\cap(\lambda-\sigma,\lambda+\sigma)\neq\emptyset.
\]

We remark that  for the above criteria to be valid,  we do not have to assume the completeness of the manifold $M$.
They can  be applied  {to} closed manifolds or open manifolds with boundary (with either Dirichlet or Neumann boundary conditions)  as well as complete noncompact manifolds.   If $M$ is compact, the criterion gives the eigenvalue estimates. \\

In most problems the ideal space to work on is the $L^2$ function space when compared to the $L^q$ spaces. However, this is not the case when considering the spectrum of the Laplacian. On a Riemannian manifold, most of the approximate eigenfunctions we can write out explicitly must be related to the distance function. It is well known however, that the Laplacian of the distance  function is locally bounded in $L^1$, but not in $L^2$.

We can see this in the following simple example.
Take $M=S^1\times (-\infty,\infty)$, letting $(\theta,x)$ be the coordinates. Then the {radial}  function with respect to the point $(0,0)$ is given by
\[
r=\sqrt{x^2+(\min (\theta, 2\pi-\theta))^2}.
\]
A straightforward computation gives
\[
\Delta r=-\frac{2\pi}{\sqrt{x^2+\pi^2}}\delta_{\{\theta=\pi\}}+\text{ a smooth function},
\]
where $\delta_{\{\theta=\pi\}}$ is the Delta function along the submanifold $\{\theta=\pi\}$. Therefore $\Delta r$ is not locally $L^2$.

The failure of the $L^2$ integrability of the Laplacian of the distance function was one of the main difficulties in applying the classical criterion above. In fact, it was not possible to prove that the $L^2$ essential spectrum of the Laplacian on a manifold with nonnegative Ricci curvature is $[0,\infty)$ by directly computing the $L^2$ spectrum.   Additional  assumptions on the curvature and geometry of the manifold were necessary (see for example~\cites{Z,jli,C-L,donnelly-1, Esc86, EF92}).

Donnelly~\cite{donnelly-1} proved that the essential spectrum of the Laplacian is $[0,\infty)$ for manifolds of nonnegative Ricci curvature  and maximal volume growth.
J-P. Wang~\cites{Wang97}, by using the seminal theorem of K. T. Sturm~\cite{sturm}, removed the maximal volume growth condition.
Wang's result confirms the conjecture that the spectrum of manifolds with nonnegative Ricci curvature is $[0,\infty)$.
In~\cite{Lu-Zhou_2011},  Lu-Zhou gave a technical generalization of Wang's result which includes the case of  manifold of finite volume.

In this article, we introduce a new method  for  computing the spectrum of a self-adjoint operator on a Hilbert space  (see Theorem \ref{Thm.Weyl.bis-2}) which  has the following application in  the case of the Laplacian
\begin{thm}\label{thm00}
Let $M$ be a Riemannian manifold and let $\Delta$ be the Laplacian. Assume that
for $\lambda\in \R^+$, there exists a nonzero function $u$ in the domain of $\Delta$ such that
\begin{equation}\label{mnb}
\|u\|_{L^\infty}\cdot\|\Delta u+\lambda u\|_{L^1}\leq\sigma\|u\|_{L^2}^2
\end{equation}
for some positive number $\sigma>0$. Then
\[
\sigma (-\Delta)\cap (\lambda-\eps,\lambda+\eps)\neq\emptyset,
\]
where
\[
\eps=\min (1,(\lambda+1)\sigma^{1/3}).
\]
Moreover,
\[
\sigma_{\textup{ess}} (-\Delta)\cap (\lambda-\eps,\lambda+\eps)\neq\emptyset,
\]
 if
for any compact subset $K$  of $M$, there exists a nonzero function  $u$  in the domain of $\Delta$  satisfying ~\eqref{mnb} whose support is outside $K$.
\end{thm}

We expect the above result  to have many applications in spectrum theory (for example, on manifolds with corners, where the test functions are usually not smooth). In this paper, we concentrate on applying the criterion  to the computation of  the essential spectrum of complete noncompact manifolds.

Theorem \ref{thm00}   proves to be a powerful tool in expanding the set of manifolds for which the $L^2$ essential spectrum is the nonnegative real line.
In the case of shrinking Ricci solitons,  we are able to prove the following result.

\begin{thm}\label{ricci-solition}
The $L^2$ essential spectrum of a complete shrinking Ricci soliton is $[0,\infty)$.
\end{thm}

Note that no curvature assumption is needed here.

For a large class of manifolds (for example, the shrinking Ricci solitons), we are able  to control the volume growth about a fixed point,  but it is difficult to prove  the uniform volume growth. Without the uniform volume growth property,  the  theorem of Sturm does not apply and the results of
 Wang~\cite{Wang97} or Lu-Zhou~\cite{Lu-Zhou_2011} cannot be used.
Therefore, the following result may be practically useful:
\begin{thm}\label{thm11}
Let $M$ be a complete noncompact Riemannian manifold. Suppose that, with respect to a fixed point $p$, the radial Ricci curvature is asymptotically nonnegative (see Lemma~\ref{DeltarEstim}), and if the volume of the manifold is finite we additionally assume that its volume  does not decay exponentially at $p$.  Then the $L^2$ spectrum of the Laplace operator on functions is $[0,\infty)$.
\end{thm}

We shall  also use Theorem~\ref{thm00}  to modify a result of Elworthy-Wang~\cite{ElwW04} on manifolds that posses an exhaustion function (Theorem~\ref{72}).
We replace the $L^2$ norm assumption by an $L^1$ norm assumption. \\

In the last section, we show  that
it is possible to work with continuous test functions in Theorem~\ref{thm00}. By using them instead we avoid the repetitive choosing of cut-off functions.\\

The essential spectrum of the Laplacian on noncompact Riemannian manifolds is interesting and important as it reveals a lot of information about the geometry  of the manifold.
Although there are lot of interesting open problems in this direction, the authors believe that answering the following conjecture is the most  important one.

\begin{conj}Let $M$ be a complete  noncompact Riemannian manifold with Ricci curvature bounded below. Then the $L^2$ essential spectrum of the Laplacian on functions is a connected subset of the real line. In other words, the essential spectrum set is of the form $[a,\infty)$, where $a$ is a nonnegative real number.\end{conj}

As is well-known, the essential spectrum of a Schr\" odinger operator could be very complicated (cf.~\cite{rs4}*{Chapter XIII}) and it certainly need not be a connected set. For the case of the Laplacian on a complete manifold however, in all known examples the $L^2$ essential spectrum is a connected set. In this paper, we  answer the above conjecture  in some special cases. We believe that the analysis of the wave kernel is needed to answer the conjecture in full.\\

{\bf Acknowledgement.} The authors thank Rafe Mezzeo and Jiaping Wang for their interest in and discussions of the essential spectrum problem. They particularly thank David Krej\v{c}i\v{r}{\'\i}k for  the discussion on the alternative versions of  Weyl's Criterion which led  to the proof of Theorem~\ref{Thm.Weyl.bis-2}.

\section{The Weyl Criterion for Quadratic Forms}
%
Let~$H$ be a self-adjoint operator on a Hilbert space~$\Hilbert$.
The norm and inner product in~$\Hilbert$ are respectively
denoted by~$\|\cdot\|$ and $(\cdot,\cdot)$. Let $\sigma(H), \sigma_\mathrm{ess}(H)$ be the spectrum and the essential spectrum of $H$, respectively. Let $\Dom(H)$ be the domain of $H$.
The Classical Weyl criterion states that
\begin{thm}[Classical Weyl's criterion]\label{Thm.Weyl}
A nonnegative real number~$\lambda$ belongs to $\sigma(H)$ if, and only if,
there exists a sequence $\{\psi_n\}_{n \in \Nat} \subset \Dom(H)$
such that
\begin{enumerate}
\item
$
  \forall n\in\Nat, \quad
  \|\psi_n\|=1
$\,,
\item
$
  (H-\lambda)\psi_n \to 0, \text{ as } n\to\infty  \text{ in }\mathcal H.$
\end{enumerate}
Moreover, $\lambda$ belongs to $\sigma_\mathrm{ess}(H)$ of $H$ if, and only if,
in addition to the above properties
\begin{enumerate}
\setcounter{enumi}{2}
\item
$
  \psi_n \to  0 \text{ weakly as  } n\to\infty \text{ in }\mathcal H.
$
\end{enumerate}
\end{thm}

\begin{remark} The above theorem is still true if the convergence in (2) is replaced by  weak convergence, the statement of which can be found (without proof) in \cite{DDi} and later in~\cite{KK}.
This version of the Weyl criterion was applied for the first time to the Laplacian on curved Euclidean domains in~\cite{KK}.
The authors are grateful to David Krej\v{c}i\v{r}{\'\i}k for informing them of  the results.
\end{remark}

We have the following  functional analytic result, which generalizes the weak Weyl criterion. To the authors' knowledge, this result seems to be new.

\begin{thm}\label{Thm.Weyl.bis-2}
Let $f$ be a  bounded positive continuous function over $[0,\infty)$.
A nonnegative real number  $\lambda$ belongs to the spectrum $\sigma(H)$ if, and only if,
there exists a sequence $\{\psi_n\}_{n \in \Nat} \subset \Dom(H)$
such that
\begin{enumerate}
\item
$
  \forall n\in\Nat, \quad
  \|\psi_n\|=1
$\,,
\item
$
 (f(H) (H-\lambda)\psi_n, (H-\lambda)\psi_n)\to 0, \text{ as } n\to\infty \quad {and}
$
\item
$
(\psi_n, (H-\lambda)\psi_n) \to  0, \text{ as } n\to\infty.
$

\end{enumerate}
Moreover, $\lambda$ belongs to $\sigma_\mathrm{ess}(H)$ of $H$ if, and only if,
in addition to the above properties
\begin{enumerate}
\setcounter{enumi}{3}
\item
$
  \psi_n \to 0, \text{ weakly as } n\to\infty
$
\text{ in } $\mathcal H$.
\end{enumerate}%
\end{thm}

\begin{proof} Since $H$ is a densely defined self-adjoint operator, the spectral measure $E$ exists and we can write
\begin{equation}\label{decomp}
H=\int_0^\infty \lambda\, dE.
\end{equation}

\noindent Assume that $\lambda\in\sigma(H)$. Then by Weyl's criterion, there exists a sequence $\{\psi_n\}$ such that
\[
\|(H-\lambda)\psi_n\|\to 0, \quad \|\psi_n\|=1
\]
as $n\to\infty$.

\noindent We write
\[
\psi_n=\int_0^\infty d E(t)\psi_n
\]
as its spectral decomposition. Then
\[
(f(H) (H-\lambda)\psi_n, (H-\lambda)\psi_n)=\int_0^\infty f(t)(t-\lambda)^2 d\|E(t)\psi_n\|^2.
\]
Since $f$ is a bounded positive function, we have
\[
(f(H) (H-\lambda)\psi_n, (H-\lambda)\psi_n)\leq C\int_0^\infty (t-\lambda)^2 d\|E(t)\psi_n\|^2=C\|(H-\lambda)\psi_n\|^2.
\]
Similarly,
\[
(\psi_n, (H-\lambda)\psi_n)\leq C\,\|\psi_n\|\cdot\|(H-\lambda)\psi_n\|.
\]
Thus the necessary part of the theorem is proved.

Now assume that $\lambda>0$ and $\lambda\notin\sigma(H)$. Then there is a $\lambda>\eps>0$ such that $E(\lambda+\eps)-E(\lambda-\eps)=0$. We write
\begin{equation}\label{fgh}
\psi_n=\psi_n^1+\psi_n^2,
\end{equation}
where
\[
\psi_n^1=\int_0^{\lambda-\eps} dE(t)\psi_n,
\]
and $\psi_n^2=\psi_n-\psi_n^1$.

\noindent Then
\begin{align*}&
(f(H) (H-\lambda)\psi_n, (H-\lambda)\psi_n) \\
&=(f(H) (H-\lambda)\psi_n^1, (H-\lambda)\psi_n^1)
+(f(H) (H-\lambda)\psi^2_n, (H-\lambda)\psi_n^2)\\
& \geq c_1\|\psi_n^1\|^2+(f(H) (H-\lambda)\psi^2_n, (H-\lambda)\psi_n^2)\geq c_1\|\psi_n^1\|^2,
\end{align*}
where the positive number $c_1$ is the infimum of the function $f(t)(t-\lambda)^2$ on $[0,\lambda-\eps]$.
Therefore
\[
\|\psi_n^1\|\to 0
\]
by {\it (2)}. On the other hand, we similarly get
\[
(\psi_n, (H-\lambda)\psi_n)\geq \eps\|\psi_n^2\|^2-\lambda\|\psi_n^1\|^2.
\]
If the criteria {\it (2), (3)} are satisfied, then, by the two {estimates} above, we conclude that both $\psi_n^1, \psi_n^2$ go to zero. This   contradicts $\|\psi_n\|=1$, and the theorem is proved.

\noindent Note that for $\lambda=0$, $\psi_n^1$ is automatically zero, and the second half of the proof would give the contradiction.

\end{proof}

We  apply Theorem \ref{Thm.Weyl.bis-2} to the Laplacian on functions. In this setting two particular cases of the function $f$ will be of interest.

\begin{corT}\label{cor-Thm.Weyl.bis-2}
A nonnegative real number  $\lambda$ belongs to the spectrum $\sigma(H)$ if, and only if,
there exists a sequence $\{\psi_n\}_{n \in \Nat} \subset \Dom(H)$
such that
\begin{enumerate}
\item
$
  \forall n\in\Nat, \quad
  \|\psi_n\|=1
$\,,
\item
$
 ((H+1)^{-1}\psi_n, (H-\lambda)\psi_n)\to 0, \text{ as } n\to\infty \quad {and}
$
\item
$
(\psi_n, (H-\lambda)\psi_n) \to  0, \text{ as } n\to\infty.
$

\end{enumerate}
Moreover, $\lambda$ belongs to $\sigma_\mathrm{ess}(H)$ of $H$ if, and only if,
in addition to the above properties
\begin{enumerate}
\setcounter{enumi}{3}
\item
$
  \psi_n \to 0, \text{ weakly as } n\to\infty
$
\text{ in } $\mathcal H$.
\end{enumerate}%
\end{corT}

\begin{proof} We take $f(x)=(x+1)^{-1}$. The  corollary follows from the identity
\[
(H+1)^{-1}(H-\lambda)=1-(\lambda+1)(H+1)^{-1}.
\]

\end{proof}

In a similar spirit, taking $f(x)=(x+\alpha)^{-(N+1)}$ for a natural number $N$ and a positive number $\alpha>1$, we also obtain the following generalization
\begin{corT}\label{cor-Thm.Weyl.bis-3}
A nonnegative real number  $\lambda$ belongs to the spectrum $\sigma(H)$ if, and only if,
there exists a sequence $\{\psi_n\}_{n \in \Nat} \subset \Dom(H)$
such that
\begin{enumerate}
\item
$
  \forall n\in\Nat, \quad
  \|\psi_n\|=1
$\,,
\item
$
 ((H+\alpha)^{-i}\psi_n, (H-\lambda)\psi_n)\to 0 \text{ as } n\to\infty \ $ for two consecutive natural numbers i=N, N+1, and
\item
$
(\psi_n, (H-\lambda)\psi_n) \to  0, \text{ as } n\to\infty.
$

\end{enumerate}
Moreover, $\lambda$ belongs to $\sigma_\mathrm{ess}(H)$ of $H$ if, and only if,
in addition to the above properties
\begin{enumerate}
\setcounter{enumi}{3}
\item
$
  \psi_n \to 0, \text{ weakly as } n\to\infty
$
\text{ in } $\mathcal H$.
\end{enumerate}%
\end{corT}

\begin{remark}
Using the Cauchy inequality, the above two corollaries reduce to  Donnelly's criterion~\eqref{sdf} when we consider the case $H=-\Delta$.
\end{remark}

\section{A spectrum estimate result}

In this section we will prove a special version of Theorem~\ref{Thm.Weyl.bis-2} for the Laplacian on functions. We begin {with}  the fact that its resolvent is always bounded on $L^\infty$.

\begin{lem} \label{lem42}
We have
\[
(-\Delta+1)^{-1}
\]
is bounded from $L^\infty(M)$ to itself and the operator norm is no more than $1$.
\end{lem}

The lemma follows from the proof of Lemma 3.1 in \cite{CharJDE}. The resolvent is bounded on $L^\infty$ because the heat kernel is bounded on $L^\infty$. This is a property that Davies proves for any nonnegative self-adjoint operator that satisfies Kato's inequality like the Laplacian ~\cite{Davies}*{Theorems 1.3.2,1.3.3}. It is due to the well-known fact that the Laplacian  on functions   is a self-adjoint operator that satisfies Kato's inequality.  {Together with Corollary~\ref{cor-Thm.Weyl.bis-2} this lemma allows us to obtain an even simpler criterion for the essential spectrum of the Laplacian on functions:}

\begin{proof}[Proof of Theorem~\ref{thm00}]
By the above lemma, we have
\begin{align*}
& |(u, (-\Delta-\lambda)u)|\leq\|u\|_{L^\infty}\cdot \|(-\Delta-\lambda)u\|_{L^1}\\
&
|((-\Delta+1)^{-1}(-\Delta-\lambda)u, (-\Delta-\lambda)u)|\leq \lambda\|u\|_{L^\infty}\cdot \|(-\Delta-\lambda)u\|_{L^1}
\end{align*}
We write
\[
u=u_1+u_2
\]
according to the spectrum decomposition of the operator $-\Delta$ (cf. ~\eqref{fgh}). Then we have
\begin{align*}
& \|u_1\|_{L^2}^2\leq \frac{\lambda(\lambda+1)}{\eps^2}\sigma\|u\|_{L^2}^2;\\
&\eps\|u_2\|_{L^2}^2-\lambda\|u_1\|_{L^2}^2\leq\sigma\|u\|^2_{L^2}.
\end{align*}
Thus we have
\[
\eps \|u_2\|_{L^2}^2+\eps\|u_1\|_{L^2}^2\leq \left(\frac{\lambda(\lambda+1)(\lambda+\eps)}{\eps^2}+1\right)\sigma\|u\|_{L^2}^2.
\]
The conclusion follows since $\|u_2\|_{L^2}^2+\|u_1\|_{L^2}^2=\|u\|_{L^2}^2$.

The essential spectrum result of the theorem follows from the classical Weyl criterion (Theorem~\ref{Thm.Weyl}, (3)).
\end{proof}

\section{An Approximation Theorem}
Let $M$ be a complete noncompact Riemannian manifold. Let $p\in M$ be a fixed point. Define
\[
r(x)=d(x,p)
\]
to be the radial function on $M$.
It is well known that
\begin{enumerate}
\item $r(x)$ is continuous;
\item $|\nabla r(x)|=1$ almost everywhere and $r(x)$ is a Lipschitz function;
\item $\Delta r$ exists on $M\backslash \{p\}$  in the sense of distribution.
\end{enumerate}

In general, it is not possible to find smooth approximations of a Lipschitz function under the  $\mathcal C^1$ norm. The following Proposition, which is a more precise version of ~\cite{Lu-Zhou_2011}*{Proposition 1}, implies that this can be done  up to a function with small $L^1$ norm.
Such kind of result is essential in Riemannian geometry and should be well-known, but  given that we were not able to find a reference, we include  a proof.

\begin{prop}\label{pp-1}
For any positive continuous decreasing function $\eta: \mathbb R^+\to\mathbb R^+$ such that
\[
\lim_{r\to\infty} \eta(r)=0,
\]
 there exist  $C^\infty$ functions $\tilde r(x)$ and $b(x)$ on $M$ such that
\begin{enumerate}
 \item[(a).] $\|b\|_{L^1(M\backslash B_{{p}}(R))}\leq\eta(R-1)$;
 \item[(b).] $\|\nabla \tilde r-\nabla r\|_{L^1(M\backslash B_{{p}}(R))}\leq\eta(R)$
\end{enumerate}
and for any $x\in M$ with $r(x)>2$
\begin{enumerate}
 \item[(c).]  $|\tilde r(x)-r(x)|\leq \eta(r(x))$ and $|\nabla\tilde r (x)|\leq 2$;
 \item[(d).] $\Delta\tilde r(x)\leq \max_{y\in B_x(1)} \Delta r(y)+\eta(r(x))+|b(x)|$  {in the sense of distribution}.
\end{enumerate}

 \end{prop}

\begin{proof}
 Without loss of generality, we assume that $\eta(r)<1$.
 Let $\{U_i\}$ be a locally finite cover of $M$ and let $\{\psi_i\}$ be the partition of unity subordinate to the cover. Let ${\bf x_i}=(x_i^1,\cdots,x_i^n)$ be the local coordinates of $U_i$. Define $r_i=r|_{U_i}$.

Let $\xi(\bf x)$ be a non-negative smooth function on $\R^n$ whose support is within the unit ball. Assume that
 \[
 \int_{\mathbb R^n}\xi=1.
 \]
 Without loss of generality, we assume that each $U_i$ is an open subset of the unit ball of $\mathbb R^n$ with coordinates ${\bf x_i}$. Then for any $\eps>0$,
  \[
 r_{i,\eps}=\frac{1}{\eps^n}\int_{\mathbb R^n}\xi\left(\frac{{\bf x_i}-{\bf y_i}}{\eps}\right)r_i({\bf y_i}) d{\bf y_i}
 \]
 is a smooth function on $U_i$ and hence on $M$. Let $\{\sigma_i\}$ be a sequence of positive numbers
 such that
 \begin{equation}
 \sum_i\sigma_i(|\Delta\psi_i(x)|+4|\nabla\psi_i (x)|+\psi_i(x))\leq\eta(r(x)).
 \end{equation}
By ~\cite{gt}*{Lemma 7.1, 7.2}, for each $i$, we can choose $\eps_i<1$ small enough so that
 \begin{align}\label{oio}
 \begin{split}
 &|r_{i,\eps_i}(x)-r_i(x)|\leq\sigma_i;\\
 &\|\nabla r_{i,\eps_i}-\nabla r_i\|_{L^1{(U_i)}}\leq\sigma_i.
 \end{split}
 \end{align}
We also have
\begin{equation}\label{oio-2}
 \Delta r_{i,\eps_i}(x)\leq \max_{y\in B_x(1)}\Delta r_i(y).
  \end{equation}

Define
 \[
 \tilde r=\sum_i\psi_i r_{i,\eps_i}, \quad b=2 { \sum_i}  \nabla\psi_i\cdot \nabla r_{i,\eps_i}.
 \]
Since $\sum_i (\nabla\psi_i\cdot \nabla r_i)=(\sum_i\nabla \psi_i) \cdot \nabla r=0$ {almost everywhere on $M$},
we have  \[
b=2\sum_i \nabla\psi_i\cdot (\nabla r_{i,\eps_i}-\nabla r_i)
\]
almost everywhere. Thus {\it (a)} follows. Similarly, observing that
\[
\tilde r-r=\sum_i\psi_i(r_{i,\eps_i}-r_i), {\quad{\rm and}\ \ |\nabla r_{i,\eps_i}|<2,}
\]
we obtain {\it (b), (c)}.

To prove {\it (d)}, we compute
\[
\Delta\tilde r=\sum_i[(\Delta\psi_i)\, r_{i,\eps_i}+2\nabla\psi_i\nabla r_{i,\eps_i}+\psi_i\Delta r_{i,\eps_i}],
\]
and
since
\[
\sum_i(\Delta\psi_i ) r_i=\sum_i (\Delta\psi_i) r=0,
\]
we have
\[
\Delta\tilde r=\sum_i[\Delta\psi_i(r_{i,\eps_i}- r_i)+b+\psi_i\Delta r_{i,\eps_i}].
\]
By~\eqref{oio-2}, we obtain {\it (d)} and the Proposition is proved.

\end{proof}

\section{Manifolds with $\Delta r$ Asymptotically Nonpositive}
{As we have mentioned in the previous section,} the Laplacian of the radial function {$r(x)=d(x,p)$} exists in the sense of distribution (except at $p$). That is, for any nonnegative smooth function $f$ with compact support in $M\backslash\{p\}$, the integral
\[
\int_Mf \Delta r
\]
{is defined}. The following simple observation  is due to Wang~\cite{Wang97} and is crucial in our estimates.

\begin{lem}\label{wang}
The function $\Delta r$ is locally integrable away from $p$.
\end{lem}

\begin{proof}
Let $W$ be any compact set of the form $B_p(R)-B_p(r)$  {for $R>r>0$}. Then by the Laplacian comparison theorem, there  is a constant $C$, depending only on the dimension,  $r$, $R$, and the lower bound of the Ricci curvature on $B_p(R)$, such that
\[
\Delta r\leq C
\]
on $W$ in the sense of distribution.  Thus we have
\[
|\Delta r|=|C-\Delta r-C|\leq 2C-\Delta r
\]
and therefore
\[
\int_W|\Delta r|\leq 2C\,\vol(W) -\int_W\Delta r.
\]
Using  Stokes' Theorem, we obtain
\[
\int_W |\Delta r|\,\leq 2C\,\vol(W)-\int_{\pa W}\frac{\pa r}{\pa n}\leq 2C\vol(W)+\vol(\pa W),
\]
and the lemma is proved.

\end{proof}

In this section, we study manifolds with the following property
\begin{equation} \label{DeltarAsy}
\overline{\lim_{r\to \infty}}\; \Delta r\leq 0
\end{equation}
in the sense of distribution, where $r(x)$ is the radial distance of $x$ to a fixed point $p$.
We shall give a precise estimate of the $L^1$ norm of $\Delta r$ in terms of the volume growth of the manifold.
But before we do that,
we first provide an important example where the above technical condition holds.

We note that for a fixed point $p\in M$ the cut locus  ${\rm Cut}(p)$  is a set of measure zero in $M$. The manifold can be written as the disjoint union $M=\Omega\cup {\rm Cut}(p)$, where $\Omega$ is star-shaped with respect to $p$. That is, if $x\in \Omega$, then the geodesic line {segment} $\overline {px}\subset \Omega$. $\p r= \p /\p r$ is well defined on  $\Omega$. We have the following result

\begin{lem}\label{DeltarEstim} Let $r(x)$ be the radial function with respect to $p$.
Suppose that there exists a continuous function $\delta(r)$ on $\mathbb{R}^+$ such that
\begin{enumerate}
\item [(i).] ${\displaystyle \lim_{r\to\infty} \delta(r)=0}$
\item [(ii).] $\delta(r)>0$ and
\item[(iii).] ${\rm Ric}(\p r, \p r)\geq -(n-1) \delta(r)$ on $\Omega$.
\end{enumerate}
Then
~\eqref{DeltarAsy} is valid
in the sense of distribution.
\end{lem}

\begin{proof}
On $\Omega$, we have the following Bochner formula
\[
0=\frac 12\Delta|\nabla r|^2=|\nabla^2r |^2+\nabla r\cdot \nabla(\Delta r)+{\rm Ric}(\pa r,\pa r).
\]
Since $\nabla ^2 r(\pa r, \pa r)=0$,
 using the Cauchy inequality, we have
\begin{equation}\label{bochner}
0\geq\frac 1{n-1}(\Delta r)^2+\frac{\pa}{\pa r}(\Delta r) -(n-1)\delta(r).
\end{equation}
Since $\Omega$ is star-shaped, for any fixed direction $\pa/\pa r$, we obtain~\eqref{DeltarAsy} by comparing the above differential inequality with the  Riccati equation.

On the points where $r$ is not smooth, we may use the trick of Gromov {as in Proposition 1.1 of~\cite{SchoenYau_bk} to conclude the result in the sense of distribution.}
\end{proof}

\subsection{Volume comparison theorems}

Let $p$ be the fixed point of the manifold. Denote
\[
B(r)=B_p(r),\quad V(r)=\vol(B_p(r))
\]
the geodesic ball of radius $r$ at $p$ and its volume respectively.

The following volume comparison theorem is well-known.

\begin{lem} \label{LemSubexp} Let $r(x)$ be the radial function defined above.  Assume that ~\eqref{DeltarAsy} is valid in the sense of distribution. Then the manifold has subexponential volume growth at $p$. In other words, for all $\eps>0$ there exists a positive constant $C(\eps)$,  depending only on $\eps$ and the manifold, such that for all $R>0$
\[
V(R)\leq C(\eps)  \, e^{\eps\, R}.
\]
\end{lem}

\begin{proof}
Let $m(r)$ be a nonnegative continuous function such that
\[
\lim_{r\to\infty} m(r)=0,
\]
and
\[
\Delta r\leq m(r)
\]
in the sense of distribution. It follows that
\[
\int_{B(R)\backslash B(1)}\Delta r\leq \int_{B(R)\backslash B(1)} m(r)
\]
which, by Stokes' Theorem, implies that
\[
{\rm vol}\,(\partial B(R))-{\rm vol}\,(\partial B(1))\leq \int_{B(R)\backslash B(1)} m(r).
\]

 Let $\eps>0$.  Then we can find an $R_\eps$ such that $m(r)<\eps$ for $r>R_\eps$.
Setting $f(R)=V(R)$, we obtain
 \[
 f'(R)\leq {\rm vol}\,(\partial B(1))+\int_{B(R_\eps)\backslash B(1)} m(r)    +\eps (f(R)-f(R_\eps))
 \]
 for any $R>R_\eps$.
 Thus
 \[
 (e^{-\eps R}(f(R)-f(R_\eps)))'\leq C_\eps e^{-\eps R}
 \]
 for $R>R_\eps$, where $C_\eps$ is a constant depending on $\eps$ and the manifold $M$. Integrating from $R_\eps$ to $R$, we obtain
 \[
 f(R)< f(R_\eps)+C_\eps\eps^{-1}e^{-\eps R_\eps}e^{\eps R}
 \]
 for $R>R_\eps$. Thus for any $R$, we have
 \[
 V(R)=f(R)<C(\eps) e^{\eps R}
 \]
 for
 \[
 C(\eps)=f(R_\eps)+C_\eps\eps^{-1}e^{-\eps R_\eps}.
 \]
\end{proof}

In other words, whenever the  Laplacian of the radial function $r(x)=d(x,p)$ is asymptotically nonnegative in the sense of distribution,  the manifold has subexponential volume growth with respect  to the point $p$.  In the case of finite volume for the manifold $M$, we will also need an assumption on the decay rate of the volume of a ball of radius $r$.  We say that the volume of $M$  {\it decays exponentially at $p$}, if there exists an $\eps_o>0$ such that
\[
\vol(M)-V(r)\leq  e^{-\eps_o r}
\]
for $r$ large. For the purposes of computing the $L^2$ essential spectrum, we will need that the volume does not decay exponentially.

\subsection{$L^1$ estimates for $\Delta \tilde{r}$.}

{We set $\tilde{r}$ to be the smoothing of $r$ from Proposition \ref{pp-1}}. The following lemma is a more precise version of ~\cite{Lu-Zhou_2011}*{Lemma 2}.

\begin{lem} \label{DeltarVolEst}
\noindent Let $r(x)$ be the radial function to a fixed point $p$ on $M$, and suppose that~\eqref{DeltarAsy} is valid in the sense of distribution.  Then we have the following two cases
\begin{enumerate}
\item [(a)]
Whenever  $\vol(M)$ is infinite,
 for any  $\eps>0$ and   $r_1>0$ {large enough},  there exists a $K=K(\eps, r_1)$ such that for any $r_2>K$,
we have
\begin{equation}\label{wsx-1}
\int_{B(r_2)\setminus B(r_1)} |\Delta \tilde{r}| \leq \eps \, V(r_2+1) ;
\end{equation}
\item [(b)]
Whenever  $\vol(M)$ is finite,  for any $\eps>0$ there exists a $K(\eps)>0$  such that for any $r_2>K$, we have
\[
\int_{M \setminus B(r_2)} |\Delta \tilde{r}| \leq  \eps \, (\vol(M)-V(r_2))+2\vol(\pa B(r_2)).
\]
\end{enumerate}
\end{lem}
\begin{proof} By Proposition ~\ref{pp-1} and using the idea in the proof of Lemma~\ref{wang}, we obtain
\[
|\Delta \tilde r(x)|\leq 2(\max_{y\in B_x(1)} \Delta r(y)+\eta(r(x))+|b(x)|)-\Delta\tilde r(x)
\]
in the sense of distribution. Using our assumptions on $\Delta r$ and $\eta$, we see that  for any $\eps>0$ we can find an $r_1>0$ large enough such that whenever $r(x)>r_1$, then
\[
2(\max_{y\in B_x(1)} \Delta r(y)+\eta(r(x))\,)<\eps/2
\]
{also in the sense of distribution.}
Therefore  for $r>r_1+2$,
\[
\int_{B(r)\setminus B(r_1)} |\Delta \tilde{r}|
\leq\frac\eps 2\, (V(r)-V(r_1))+2\int_{M\setminus B(r_1)}|b|-\int_{B(r)\setminus B(r_1)}\Delta\tilde r.
\]
Using Stokes' Theorem, we get
\[
\int_{B(r)\setminus B(r_1)}|\Delta \tilde{r}|\leq \frac\eps 2\,  (V(r)-V(r_1))+2\int_{M{\setminus B(r_1)}} |b|-\int_{\p B(r)}\frac{\p \tilde{r}}{\p n} + \int_{\p B(r_1)}\frac{\p \tilde{r}}{\p n},
\]
where $\p/\p n$ is the outward normal direction on the boundary. Obviously, {the above implies that}
\begin{align} \label{qaz-3}
\int_{B(r)\setminus B(r_1)}|\Delta \tilde{r}| \leq & \frac\eps 2\,  (V(r)-V(r_1))+2\int_{M{\setminus B(r_1)}}|b| \\
&+\int_{\p B(r)}\left|\frac{\p \tilde{r}}{\p n} -1\right|+ \int_{\p B(r_1)}\frac{\p \tilde{r}}{\p n}.\notag
\end{align}

We first consider the case when the volume of $M$ is infinite. By Proposition  ~\ref{pp-1}, choosing $r_1$ large enough we obtain
\[
\int_{M{\setminus B(r_1)}}|b|< \frac{\eps}{4}
\]
and
\begin{equation}\label{qaz-2}
\|\nabla \tilde r-\nabla r\|_{L^1(M\backslash B(r_1))}\leq 1
\end{equation}

Since the volume of $M$ is infinite, then there exists $K=K(\eps, r_1)>r_1+2$ such that whenever $r>K$
\begin{equation}\label{qaz-1}
\int_{B(r)\setminus B(r_1)}|\Delta \tilde{r}|\leq \frac{3\eps} 4\,( V(r){-V(r_1)\,})+\int_{\pa B(r)}\left|\frac{\pa\tilde r}{\pa n}-1\right|.
\end{equation}

We choose an $r'$ such that $|r'-r|<1$ and
\[
\int_{\pa B(r')}\left|\frac{\pa\tilde r}{\pa n}-1\right|\leq \int_{r-1}^{r+1}\int_{\pa B(t)}\left|\frac{\pa\tilde r}{\pa n}-1\right|dt.
\]
By~\eref{qaz-2}, we have
\[
\int_{\pa B(r')}\left|\frac{\pa\tilde r}{\pa n}-1\right|<2.
\]
Therefore,
\[
\int_{B(r')\setminus B(r_1)}|\Delta \tilde{r}|\leq \frac{3\eps}{4}\,  (V(r')-V(r_1))+2.
\]
Choosing a possibly larger $K(\eps,r_1)$ we get  {\it (a)}.

The proof of {\it (b)} is similar. We choose $\eta(r)$ decreasing to zero so fast so that

\[
\int_{M\backslash B(r_1)}|b|\leq \frac\eps 8(\vol (M)- V(r_1)).
\]
Since the volume of $M$ is finite, { sending $r\to\infty$ in  \eref{qaz-3}} we have
\[
\int_{M\setminus B_p(r_1)}|\Delta \tilde{r}|\leq \eps\,  (\vol(M)-V(r_1))+ \int_{\p B_p(r_1)}\frac{\p \tilde{r}}{\p n}.
\]
Since $|{\pa \tilde r}/{\pa n}|\leq 2$ by {\it (c)} of Proposition~\ref{pp-1},  the lemma follows.

\end{proof}

\begin{corT} \label{corlDeltar} Suppose that $(i), (ii), (iii)$ hold on $M$ as in Lemma \ref{DeltarEstim}. Then the same integral estimates for $\Delta \tilde{r}$ hold as in Lemma \ref{DeltarVolEst}.
\end{corT}

\section{The $L^2$ Spectrum.}

 \label{sec5}

In this section,{we let $\tilde r(x)$ be the smoothing function defined in Proposition~\ref{pp-1}} of the radial function $r(x)=d(x,p)$.
For each $i\in\mathbb N$, let $x_i, y_i, R_i, \mu_i$ be large positive numbers such that $x_i>2R_i>2\mu_i+4$ and $y_i>x_i+2R_i$. We take the cut-off functions $\chi_i: \mathbb{R}^+\to \mathbb{R}^+$, smooth with support on $[x_i/R_i-1, y_i/R_i+1]$ and such that $\chi_i=1$ on $[x_i/R,y_i/R]$ and  $|\chi'_i|, |\chi''_i|$  bounded. Let $\lambda>0$ be a positive number. We let
\begin{equation}\label{4-equ}
 \phi_i(x)=\chi_i({\tilde{r}}/{R_i})\, e^{\sqrt{-1}\sqrt{\lambda}\,\tilde{r}}.
\end{equation}
Setting $\phi=\phi_i$, $R=R_i$, $x=x_i$ and $\chi=\chi_i$, we compute
\begin{align*}
\begin{split}&
\Delta\phi +\lambda\phi  =
(R^{-2}\chi''(\tilde r/R)+2i\sqrt\lambda R^{-1}\chi'(\tilde r/R))e^{\sqrt{-1}\sqrt\lambda\tilde r}|\nabla\tilde r|^2\\&
-\lambda\phi(|\nabla\tilde r|^2-1)+(R^{-1}\chi'(\tilde r/R)+i\sqrt{\lambda}\chi)e^{\sqrt{-1}\sqrt\lambda\tilde r}\Delta\tilde r.
\end{split}
\end{align*}

Then we have
\begin{equation} \label{2}
|\phi|\leq 1,\quad |\Delta\phi +\lambda\phi|\leq \frac{C}{R}+C|\Delta\tilde r| +C|\nabla\tilde r-\nabla r|,
\end{equation}
{where $C$ is a constant depending only on $\lambda$ and $M$.}

Denote the inner product on $L^2(M)$ by $( \cdot\,,\,\cdot)$. We have the following key estimates
\begin{lem}\label{lem41} Suppose that~\eqref{DeltarAsy} is valid for the radial function $r$ in the sense of distribution. In the case that the volume of $M$ is finite, we make the further assumption that its volume does not decay exponentially at $p$. Then there exist sequences of large numbers $x_i, y_i, R_i, \mu_i$ such that the supports of the $\phi_i$ are disjoint and
\[
\frac{\|(\Delta+\lambda)\phi_i\|_{L^1}}{(\phi_i,\phi_i)}\to 0
\]
as $i\to\infty$.
\end{lem}

\begin{proof}   {The   proof    is similar to that of ~\cite{Lu-Zhou_2011}.}
We define $x_i,y_i,R_i,\mu_i$ inductively. If $(x_{i-1}, y_{i-1}, R_{i-1}, \mu_{i-1})$ are defined, then we only need to let $\mu_i$ large enough so that the support of $\phi_i$ is disjoint with the previous $\phi_j$'s.
For simplicity  we suppress the $i$ in our notation. The upper bound estimates for $|\phi|$ and $|\Delta\phi+\lambda\phi|$ given in \eref{2}imply that
\begin{align} \label{1}
\begin{split}
\int_M(\phi, \Delta\phi+\lambda\phi)  \leq  & \frac CR \,[V(y+R) -V(x-R)] \\
& +  C \int_{B(y+R)\setminus B(x-R)}   |\Delta \tilde{r}| +\eta(x-R).
\end{split}
\end{align}
When the volume of $M$ is infinite, we choose a function $\eta$ as in Proposition \ref{pp-1} such that $\eta\leq 1$. By Lemma \ref{DeltarVolEst}, if we choose $R, x$ large enough
but fixed,   then for any $y>0$  large enough  we have
\[
\int_M(\phi, \Delta\phi+\lambda\phi) \leq  2 \eps  \,V(y+R+1).
\]
Since $\|\phi\|_2^2\geq V(y)-V(x)$, if we choose $y$ large enough, $\|\phi\|_2^2\geq \frac 12 V(y).$  The subexponential volume growth of $M$ at $p$ that was proved in Lemma \ref{LemSubexp} implies that there exists a sequence of $y_k\to \infty$ such that $V(y_k+R+1)\leq 2\, V(y_k)$. If not, then for a fixed number $y$ and for all $k\in \mathbb{N}$ we have that
\[
V(y+k(R+1))> 2^k\, V(y).
\]
However, by the subexponential volume growth of the manifold
\[
2^k\, V(y) < V(y+k(R+1))\leq C(\eps_1) \, e^{\eps_1 y} \, e^{k \,\eps_1 (R+1)}
\]
for any $\eps_1>0$  and $k$ large. This leads to a contradiction when we choose $\eps_1$ such that $\eps_1 R  <\log 2$.
Therefore, there exists a $y$  such that
\[
V(y+R +1)\leq 2\,V(y)\leq 4 \|\phi\|_2^2.
\]
Combing the above inequalities, we have
\[
\int_M(\phi, \Delta\phi+\lambda\phi) \leq   8 \eps \|\phi\|_2^2.
\]

We now consider the finite volume   case. Using equation \eref{1} and Lemma \ref{DeltarVolEst} we obtain for $x-R>K(\eps)$
\begin{align*}
\int_M(\phi, \Delta\phi+\lambda\phi) \leq & (R^{-1} +  \eps) \,[\vol(M) -V(x-R)]\\
& + 2 C \,\vol(\p B(x-R))
 +\eta(x-R).
\end{align*}
We set $h(r)=\vol(M) -V(r)$, a {decreasing} function. We choose $\eta(r)$  as in Proposition \ref{pp-1}
 so that $\eta(r)\leq \frac\eps 8 h(r)$. Making $\eps$ even smaller and choosing $R$ and $x-R$ large enough, we get\[
\int_M(\phi, \Delta\phi+\lambda\phi) \leq  \eps \,h(x-R) - 2 C \,h'(x-R).
\]
Given that $\|\phi\|_2^2\geq h(x)-h(y)$ and the volume of $M$ is finite, we can choose $y$ large enough so that \[
\|\phi\|_2^2\geq \frac 12 h(x).
\]

We would like to prove in this case that there exists a sequence of $x_k\to \infty $ such that
\[
\eps \,h(x_k-R) - 2C\,h'(x_k-R)\leq 2 \eps h(x_k).
\]
If the above inequality does not hold, then for all $x$ large enough
\[
\eps \,h(x-R) -2C \,h'(x-R) > 2 \eps h(x).
\]
Replacing $\eps$ by $\eps/2C$, we obtain
\[
\eps \,h(x-R) - \,h'(x-R) > 2 \eps h(x).
\]
This implies that
\[
- \bigl(e^{-\eps x}h(x-R)\bigr)' > 2 \eps   h(x) \, e^{-\eps x}.
\]
Integrating from $x$ to $x+R$ and using the monotonicity of $h$ we have
\[
h(x-R) >2(1 -e^{-\eps R}) h(x+R).
\]
Choosing $R$ even larger, we can make  $2 {(1-e^{-\eps R})}> 5/4$, therefore
\[
h(x-R)  >  \frac 54 h(x+R)
\]
for all $x$ large enough. By iterating this inequality, we get for all positive integers $k$
\[
h(x-R) > \left(\frac 54\right)^k\, h(x+(2k-1)R).
\]
Therefore
\[
\vol(M)-V(x-R) >   \left(\frac 54\right)^k [\vol(M)-V(x+(2k-1)R)\, ]
\]
which gives
\[
\vol(M)-V(x+(2k-1)R) \leq  \left(\frac 45\right)^k \, \vol(M).
\]
Sending $k\to \infty$ this contradicts the nonexponential decay assumption on the volume.
\end{proof}

Corollary \ref{corlDeltar} gives
\begin{corT} \label{corl41}
Suppose that $(i), (ii), (iii)$ hold on $M$ as in Lemma \ref{DeltarEstim}. In the case that the volume of $M$ is finite, we make the further assumption that its volume does not decay exponentially at $p$. Then there exist sequences of large numbers $x_i, y_i, R_i, \mu_i$
and cut-off functions $\chi_i$
such that the supports of the $\phi_i$ are disjoint and
\[
\frac{\|(\Delta+\lambda)\phi_i\|_{L^1}}{(\phi_i,\phi_i)}\to 0
\]
as $i\to\infty$.

\end{corT}

Now  we prove Theorem~\ref{thm11}.
In fact we will be able to prove a more general, albeit more technical result

\begin{thm}\label{thm12}
Let $M$ be a complete noncompact Riemannian manifold. Suppose that, with respect to a fixed point $p$, the radial function $r(x)=d(x,p)$ satisfies
\begin{equation*}
\overline{\lim_{r\to \infty}}\; \Delta r\leq 0
\end{equation*}
in the sense of distribution, and if the volume of the manifold is finite, we additionally assume  that its volume  does not decay exponentially at $p$.  Then the $L^2$ spectrum of the Laplace operator on functions is $[0,\infty)$.
\end{thm}

\begin{proof}
Let $\phi_i$ be the sequence of functions as defined in ~\eqref{4-equ}.  Then by the construction of the functions
and Corollary~\ref{corl41}, the assumptions of Theorem~\ref{thm00} are satisfied. This completes the proof of the theorem.\end{proof}

\begin{remark} \label{rmkwarp} We note that a similar result should hold on warped product manifolds $M=\mathbb{R}\times_J \tilde{M}$ with metric $g=d\rho^2 + J^2(\rho,\theta) \, \tilde{g},$ where $(\tilde{M},\tilde{g})$ is a compact $(n-1)$-dimensional submanifold of $M$ and $\rho$ is the distance function from this submanifold.  Under the same asymptotically nonnegative assumption on ${\rm Ric}(\p \rho, \p \rho)$ as in Lemma~\ref{DeltarEstim}, we also get that the $L^2$ spectrum of the Laplace operator on functions is $[0,\infty)$.
\end{remark}

\section{Complete Shrinking Ricci Solitons}

A noncompact  complete Riemannian manifold  $M$ with metric $g$ is  called a gradient shrinking Ricci soliton if there exists a smooth function $f$ such that  the Ricci tensor of the metric $g$ is given by
\[
R_{ij}+\n_i \n_j f = \rho \, g_{ij}
\]
for some positive constant $\rho>0$. By rescaling the metric we may rewrite the soliton equation as
\[
R_{ij}+\n_i \n_j f = \frac 12 \, g_{ij}.
\]
The scalar curvature $R$ of a gradient shrinking Ricci soliton is nonnegative, and the volume growth of such manifolds (with respect to the Riemannian metric) is Euclidean.  Hamilton~\cite{Ham} proved that the scalar curvature of a gradient shrinking Ricci soliton satisfies the equations
\[
\n_i R  =  2 \, R_{ij}\,  \n_j f,
\]
\[
R+|\n f|^2-f=C_o
\]
for some constant $C_o$. We may add a constant to $f$ so that
\[
R+|\n f|^2-f=0.
\]

In ~\cite{Lu-Zhou_2011}, the authors proved that
\begin{enumerate}
\item the $L^1$ essential spectrum of the Laplacian contains $[0, \infty)$;
\item the $L^2$ essential spectrum of the Laplacian is $[0, \infty)$, if the scalar curvature has sub-quadratic growth.
\end{enumerate}
Using our new Weyl Criterion, we  are able to remove the curvature condition.

\begin{proof}[Proof of Theorem~\ref{ricci-solition}]
It can be shown that $f(x)\geq 0$ and the key idea is to use $\rho(x)=2\sqrt{f(x)}$ as an approximate distance function on the manifold, because of the special properties that it satisfies.

We define
\[
D(r)=\{x\in M : \rho(x)<r\}
\]
and set $V(r) =\vol(D(r))$.
\noindent For some positive number $y$ sufficiently large we consider the cut-off function $\chi: \mathbb{R}^+\to \mathbb{R}$, smooth with support in $[0, y +2]$ and such that $\chi=1$ on $[1,y+1]$ and  $|\chi'|, |\chi''|\leq C$. For any $\lambda > 0$  and large enough constants $b, l$ we let
\begin{equation*}
 \phi(\rho)=\chi\left(\frac{\rho-b}{l}\right)\, e^{\sqrt{-1}\sqrt{\lambda}\,\rho}
\end{equation*}
{which has support on $[b+l, b+l(y+1)]$}.  Lu and Zhou~\cite{Lu-Zhou_2011}*{page 3289} demonstrate that for sufficiently large $l$ and $b$
\[
\int_M |\Delta \phi +\lambda \phi| \leq \eps V(b+(y+2)l).
\]
At the same time
\[
\|\phi\|_{L^2}^2\geq V(b+(y+1)l)-V(b+l)
\]
(note that the same holds true for the $L^1$ norm of $\phi$). Arguing as is ~\cite{Lu-Zhou_2011}*{Theorem 6} we conclude that there exists a  $y$ large enough such that
\[
\int_M |\Delta \phi +\lambda \phi| \leq 4 \eps \|\phi\|_{L^2}^2.
\]
As in the previous section, we may also choose appropriate sequences of $b_i, l_i$ such that the supports of the $\psi_i$ are disjoint and  condition  {\it (2)} of Theorem~\ref{thm00} holds. Condition  {\it (1)} is verified by the estimate above and the fact that $\|\phi_i\|_{L^\infty}=1$.\end{proof}

\section{Exhaustion functions on complete manifolds}

From what we have seen  so far, it is apparent that two things are important when computing the essential spectrum of the Laplacian:
\begin{enumerate}
\item The control of the $L^1$ norm of $\Delta r$;
\item The control of the volume growth and decay of geodesic balls.
\end{enumerate}

The same idea can be used for manifolds whose essential spectrum is not the half real line.

In the spirit of the results above, we are also able to modify a theorem of Elworthy and Wang~\cite{ElwW04}. We now consider manifolds on which there exists a continuous exhaustion function $\gamma\in \mathcal C(M)$ such that

(a) $\gamma$ is unbounded above and is $\mathcal C^2$ smooth in the domain $\{\gamma>R\}$ for some $R>0$ and

(b) $\vol(\{m_o<\gamma<n\})<\infty $ for some $m_o$ and any $n>m_o$ where the volume is measured with respect to the Riemannian metric.

For $t>0$ and $c\in \mathbb{R}$ we define $B_t=\{\gamma(x)<t\}$ and set
 $dv_c=e^{-c\gamma}dv$. For $t\geq s$, let $U_c(s,t)=\vol_c(B_t\setminus B_s)$ where $\vol_c$ is the volume with respect to the measure $dv_c$.

We begin by stating the result of Elworthy and Wang for the sake of comparison.

\begin{thm}[\cite{ElwW04}*{Theorem 1.1}]
Suppose that there exists a function $\gamma \in \mathcal C(M)$ that satisfies $(a)$ and $(b)$ and a constant $c\in\mathbb{R}$ such that
\begin{equation} \label{exh2}
\lim_{s\to\infty} \overline{\lim_{t\to\infty}}\;  U_c(s,t)^{-1}\int_{B_t\setminus B_s} [(\Delta\gamma-c)^2+(|\n\gamma|^2-1)^2 ] \, dv_c=0
\end{equation}
and
\begin{equation} \label{exh1}
\lim_{t\to\infty} \max \{U_c(m_o,t), U_c(t,\infty)^{-1}  \} \,e^{-\eps t}=0 \qquad \rm{for \ any\ } \eps>0.
\end{equation}
Then $\sigma(-\Delta)\supset  [c^2/4,\infty).$ When the above hold for $c=0$, then $\sigma(-\Delta)= [0,\infty).$\\
\end{thm}

Note that condition \eref{exh1} implies that when $c = 0$ the volume of the manifold grows and decays subexponentially, as was the case for us in the previous sections. The assumption here is that the weighted volume  grows and decays subexponentially.

Our result is as follows:

\begin{thm}\label{72}
Suppose that there exists a function $\gamma \in \mathcal C(M)$ that satisfies $(a)$ and $(b)$ and a constant $c\in\mathbb{R}$ such that
\begin{equation} \label{exh3}
\lim_{s\to\infty} \overline{\lim_{t\to\infty}}\; \,U_{c}(s,t)^{-1} \,\int_{B_t\setminus B_s} (|\Delta\gamma-c|+|\,|\n\gamma|^2-1|) \, dv_{c}=0
\end{equation}
and
\begin{equation}\label{exh4}
\lim_{t\to\infty} \max \{U_c(m_o,t), U_c(t,\infty)^{-1}  \} \,e^{-\eps t}=0 \qquad \rm{for \ any\ } \eps>0.
\end{equation}
If \eref{exh3} and \eref{exh4} hold for $c=0$, then  $\sigma(-\Delta)=[0,\infty).$

In the case they hold for $c\neq 0$, we make the additional assumptions that the heat kernel of the Laplacian satisfies the pointwise bound
\begin{equation}\label{exh5}
p_t(x,y)\leq C t^{-m}\,e^{-\frac{(\gamma(x)-\gamma(y))^2}{4C_1 t} -\frac{d(x,y)^2}{4C_2 t } +\beta_1|\gamma(x)-\gamma(y)|+\beta_2d(x,y)+\beta_3 t}
\end{equation}
for some positive constants $m, C_1, C_2, \beta_1, \beta_2, \beta_3$, and that the Ricci curvature of the manifold is bounded below ${\rm Ric}(M)\geq -(n-1)K$ for a nonnegative number $K$. Then $\sigma(-\Delta) \supset [c^2/4,\infty).$
\end{thm}

In the case $c=0$, the main difference between our result and Theorem 1.1 of~\cite{ElwW04} is that we only need to control the $L^1$ norms  of $|\Delta\gamma-c|$ and  $| |\n\gamma|^2-1|$ as in \eref{exh3}, instead of their $L^2$ norms (compare to \eref{exh2}).  Our assumption is weaker in various cases, for example when $\gamma$ is the radial function where we know that its Laplacian is not locally $L^2$ integrable when the manifold has a cut-locus, but it is locally $L^1$ integrable.

In the case $c\neq 0$, the additional assumption \eref{exh5} is similar to requiring a uniform Gaussian bound for the heat kernel, but now with respect to the $\gamma$ function as well. Such a bound is certainly true in the case of hyperbolic space with $\gamma$ the radial function.

The proof uses similar estimates to those of Elworthy and Wang for the measures of annuli along the exhaustion function $\gamma$. We provide an outline of the argument with the necessary  modifications.

\begin{proof}
Set $\lambda\geq c^2/4$ be a fixed number. For any $t>s$ we let $\chi: \mathbb{R}^+\to \mathbb{R}^+$, be a smooth cut-off function with support on $[s-1, t+1]$ and such that $\chi=1$ on $[s,t]$ and  $|\chi'|, |\chi''|$  bounded. Let $\lambda_c=\sqrt{\lambda-c^2/4}$ and define for $s\geq 0$
\[
f(s)= e^{(i{\lambda_c} - c/2)\,s}.
\]
Consider the test function
\[
 \phi_{s,t}(x)=\chi(\gamma(x))\, f(\gamma(x)).
\]
We compute
\[
\Delta\phi_{s,t} +\lambda\phi_{s,t}  = (\chi''f+2\chi'f'+\chi f'')|\nabla\gamma|^2
+ (\chi' f + \chi f')\Delta\gamma +\lambda \,\chi f.
\]
Using the fact that $f'' +c f' +\lambda f=0$ we obtain
\[
\Delta\phi_{s,t} +\lambda\phi_{s,t}  = (\chi''f+2\chi'f')|\nabla \gamma|^2
+ (\chi' f )\Delta\gamma + \chi f'(\Delta \gamma -c |\n \gamma|^2) +  \lambda \,\chi f(1-|\n \gamma|^2).
\]
Therefore there exists a constant $C$ such that
\begin{equation}\label{73}
|\Delta\phi_{s,t} +\lambda\phi_{s,t}|\leq
Ce^{-c/2\gamma}\, {\bigl[}(|\Delta\gamma-c|+|\,|\n\gamma|^2-1|) 1_{\text{spt}(B_{{t+1}}\setminus B_{{s-1}})} +1_{\text{spt}(\chi')} {\bigr]}.
\end{equation}

For the rest of the estimates, we will repeatedly use
\begin{equation}\label{74}
\lim_{s,t\to\infty} (U_c(s-1,s)+U_c(t,t+1))/U_c(s,t)=0,
\end{equation}
which follows from ~\eqref{exh4}.

Using~\eqref{73}, we have
\begin{align}\label{75}
|(\phi_{s,t},\Delta\phi_{s,t} +\lambda\phi_{s,t})| \leq & C \int_{B_{{t+1}}\backslash B_{{s-1}}}(|\Delta\gamma-c|+|\,|\n\gamma|^2-1|)\,dv_c \\
&{+ C (U_c(s-1,s)+U_c(t,t+1))} .\notag
\end{align}
We observe that
\begin{align*}
\frac{1}{U_c(s,t)}\int_{B_{t+1}\backslash B_{s-1}}&(|\Delta\gamma-c|+|\,|\n\gamma|^2-1|)\,dv_c \\
=&\bigl[1+\frac{U_c(s-1,s)+U_c(t,t+1)}{U_c(s,t)}\bigr] \\
&\cdot\,\frac{1}{U_c(s-1,t+1)}\int_{B_{t+1}\backslash B_{s-1}}(|\Delta\gamma-c|+|\,|\n\gamma|^2-1|)\,dv_c,
\end{align*}
which tends to zero as $s,t\to\infty$ by ~\eqref{74} and assumption \eref{exh3}. Since $\|\phi_{s,t}\|^2_{L^2}\geq U_c(s,t)$, inequality \eref{75}, the above estimate and ~\eqref{74} imply that
\begin{equation} \label{76}
\lim_{s,t\to\infty} |(\phi_{s,t}, \Delta\phi_{s,t} +\lambda\phi_{s,t})|/\|\phi_{s,t}\|^2_{L^2}=0.
\end{equation}

When $c=0$, we  choose  appropriate sequences of $s_n, t_n \to \infty$ such that condition {\it (2)} of Theorem~\ref{thm00} holds. Condition {\it (1)} of the Corollary  follows from \eref{76} and the fact that the functions $\phi_{s_n,t_n}$ are bounded. Therefore, $\lambda_0=\sqrt{\lambda}$ belongs to the essential $L^2$ spectrum. Given that $\lambda$ is any nonnegative number, the result follows.

 In the case $c\neq 0$, we will apply Corollary \ref{cor-Thm.Weyl.bis-3}.  For a fixed natural number $i>m$ and $\alpha>0$ we have that the integral kernel of $(-\Delta+\alpha)^{-i}$, $g_{\alpha}^{i}(x,y)$, is given by
\[
g_{\alpha}^{i}(x,y)= C(n)\int_0^\infty p_t(x,y)\, t^{i-1} \, e^{-\alpha t} \, dt.
\]
On the other hand, it is a property of the exponential function that for any $\beta_4, \beta_5 \in \mathbb{R}$
\[
e^{-\frac{(\gamma(x)-\gamma(y))^2}{4C_1 t}} \leq e^{-\beta_4|\gamma(x)-\gamma(y)|} e^{C_1\,\beta_4^2 t}
\]
and
\[
e^{-\frac{d(x,y)^2}{4C_2 t}} \leq e^{-\beta_5 d(x,y)} e^{C_2\,\beta_5^2 t}.
\]
Combining the above, we have that for any $N>m$ and $\beta_4, \beta_5>0$ there exists an $\alpha>0$ large enough, and a constant $C$ such that
\[
g_{\alpha}^{i}(x,y)\leq C\,e^{-\beta_4 \, |\gamma(x)-\gamma(y)|-\beta_5 d(x,y)}
\]
for $i=N, N+1$.
As a result, for any $t>s>2$
\begin{align*}
\int _{B_{t+1}\setminus B_{s-1}} g_{\alpha}^{i}(x,y) e^{-c/2\gamma(y)} dy&\leq C \int _{B_{t+1}\setminus B_{s-1}} e^{-\beta_4 |\gamma(x)-\gamma(y)| -\beta_5 d(x,y)} \,  e^{-c/2\gamma(y)} dy \\
&\leq C\,e^{-c/2\gamma(x)}
\end{align*}
after choosing $\beta_4=|c|/2$ and $\beta_5>\sqrt{K}$. This estimate together with \eref{74} also give
\begin{align*}
|(\,(-\Delta+\alpha)^{-i}\phi_{s,t},\Delta\phi_{s,t} +\lambda\phi_{s,t})| \leq &  C \int_{B_{{t+1}}\backslash B_{{s-1}}}(|\Delta\gamma-c|+|\,|\n\gamma|^2-1|)\,dv_c \\
& + C (U_c(s-1,s)+U_c(t,t+1)) .
\end{align*}

As a result,
\begin{align} \label{77}
\lim_{s,t\to\infty} |(\,(-\Delta+\alpha)^{-i}\phi_{s,t},\Delta\phi_{s,t} +\lambda\phi_{s,t})|/\|\phi_{s,t}\|^2_{L^2}=0.
\end{align}
Choosing  appropriate sequences of $s_n, t_n \to \infty$ and setting $\psi_n = \phi_{s_n,t_n}/\|\phi_{s_n,t_n}\|_{L^2}$, conditions {\it (1)} and {\it (4)} of Corollary \ref{cor-Thm.Weyl.bis-3} hold for the functions $\psi_n$. That {\it (2)} and {\it (3)} also hold follows from \eref{76} and \eref{77} respectively.

\end{proof}

\section{The use of continuous test functions.}
In this section we will see that it is not necessary to use cut-off functions in our test functions. We will do that by first proving  yet another version of the generalized Weyl's criterion (Corollary \ref{cor-Thm.Weyl.bis-8}).  This version of Weyl's Criterion  sometimes provides a cleaner method for computing the essential spectrum.

Let $D$ be a bounded domain of $M$ with smooth boundary. We use the notation $\mathcal C_0^\infty(D)$ to denote the set of smooth functions on the closure $\bar D$ which vanish on the boundary $\pa D$.
Let $\rho:D\to \mathbb R$ be the distance function to the boundary $\pa D$.

\begin{Def} \label{def51}
We define $\mathcal C_0^+( D)$ to be the set of functions $f$ on $ D$ with the properties
\begin{enumerate}
\item[$(1)$] $f$ is continuous, vanishing on $\pa D$;
\item[$(2)$] $f$ is Lipschitz, $\nabla f$ is essentially bounded, and $|\Delta f|$ exists in the sense of distribution;
\item[$(3)$] As $\eps\to 0,$ $\int_{\{\rho \leq \eps\}}|f|\leq \frac 12\eps^2(\int_{\pa D} |\nabla f|+o(1))$, and
$\int_{\{\rho \leq \eps\}}|\nabla f|\leq \eps(\int_{\pa D} |\nabla f|+o(1))$.
\end{enumerate}
\end{Def}

Let $\mathcal C_0^+(M)$ be the set of continuous functions whose support is a bounded domain of $M$ with smooth boundary and
\[
f\in \mathcal C_0^+({\rm supp}{f}).
\]

We have the following

\begin{corT}\label{cor-Thm.Weyl.bis-8}
A nonnegative real number  $\lambda$ belongs to the  spectrum $\sigma(-\Delta)$,  if
there exists a sequence $\{\psi_n\}_{n \in \Nat}$ of functions in $\mathcal C_0^+(M)$
such that
\begin{enumerate}
\item
${\displaystyle \frac{\|\psi_n\|_{L^\infty(D_n)}\cdot(\|(-\Delta-\lambda)\psi_n\|_{L^1( D_n)}+\|\nabla\psi_n\|_{L^1(\pa D_n)})}{\|\psi_n\|_{L^2(D_n)}^2} \to  0, \text{ as } n\to\infty,} $
\end{enumerate}
where $D_n= {\rm supp}\,{\psi_n}$.
Moreover, $\lambda$ belongs to $\sigma_{\rm ess}(-\Delta)$ of $\Delta$, if
\begin{enumerate}
\setcounter{enumi}{1}
\item  For any compact subset $K$  of $M$, there exists an $n$  such that the support of $\psi_n$ is outside $K$.
\end{enumerate}
\end{corT}

The above corollary can be proved using the following approximation result

\begin{prop} Let $f\in \mathcal C_0^+(M)$. Then for any $\eps>0$, there exists a smooth function $h$ of $M$ such that
\begin{enumerate}
\item[(a)] ${\rm supp}\,(h)\subset {\rm supp}\, (f)$;
\item[(b)] $\|f-h\|_{L^1}+\|f-h\|_{L^2}\leq \eps$;
\item[(c)] $\|(-\Delta -\lambda) h\|_{L^1}\leq C (\|(-\Delta-\lambda) f\|_{L^1( D)}+\|\nabla f\|_{L^1(\pa D)})$,
\end{enumerate}
where $C$ is a constant independent of  $f$,  and $D={\rm supp}\,(f)$.
\end{prop}

\begin{proof} Let $\chi(t)$ be a cut-off function such that it vanishes in a neighborhood of $0$ and is $1$ for $t\geq 1$. Let $\delta>0$ be a small number. Consider
\[
g_\delta(x)=\chi\left(\frac{\rho(x)}{\delta}\right) f(x).
\]
It is not difficult to prove {\it (a), (b)} in the Proposition when we replace $h$ by $g_\delta$. To prove {\it (c)} we compute
\[
(-\Delta-\lambda) g_\delta=\chi(-\Delta-\lambda) f-2\delta^{-1}\chi'\nabla \rho\nabla f -(\delta^{-2}\chi''+\delta^{-1}\chi'\Delta\rho) f.
\]
Since $\pa D$ is smooth, $\rho$ is a smooth function near $\pa D$. Therefore  by (3) of Definition~\ref{def51} we have
\[
\|(-\Delta-\lambda) g_\delta\|_{L^1}\leq C (\|(-\Delta-\lambda) f\|_{L^1( D)}+\|\nabla f\|_{L^1(\pa D)})
\]
for $\delta$ sufficiently small.

The proof that $g_\delta$ can be approximated by a smooth  {function}  is similar to that
of Proposition~\ref{pp-1}. We sketch the proof here.

Let $D=\cup U_i$ be a finite cover of $D$. {Without loss of generality, we assume that those $U_i$'s which intersect with $\pa D$  are outside the support of $g_\delta$.}
Let ${\bf x_i}=(x_i^1,\cdots,x_i^n)$ be the local coordinates of $U_i$. Define $g_i=g_\delta|_{U_i}$.

Let $\xi(\bf x)$ be a non-negative smooth function of $\R^n$ whose support is within the unit ball. Assume that
 \[
 \int_{\mathbb R^n}\xi=1.
 \]
 Without loss of generality, we assume that each $U_i$ is an open subset of the unit ball of $\mathbb R^n$ with coordinates ${\bf x_i}$. Then for any $\eps>0$,
  \[
 g_{i,\eps}=\frac{1}{\eps^n}\int_{\mathbb R^n}\xi\left(\frac{{\bf x_i}-{\bf y_i}}{\eps}\right)g_i({\bf y_i}) d{\bf y_i}
 \]
 is a smooth function on $U_i$ and hence on $M$. Let $\{\sigma_i\}$ be a sequence of positive numbers
 such that
 \begin{equation}
 \sum_i\sigma_i(|\Delta\psi_i(x)|+4|\nabla\psi_i (x)|+\psi_i(x))
 \end{equation}
is sufficiently small.
 By ~\cite{gt}*{Lemma 7.1, 7.2}, for each $i$, we can choose $\eps_i<1$ small enough so that
 \begin{align}\label{oio-4}
 \begin{split}
 &|g_{i,\eps_i}(x)-g_i(x)|\leq\sigma_i;\\
 &\|\nabla g_{i,\eps_i}-\nabla g_i\|_{L^1{(U_i)}}\leq\sigma_i.
 \end{split}
 \end{align}
We also have
\begin{equation}\label{oio-6}
 \|\Delta g_{i,\eps_i}\|_{L^1}\leq  \|\Delta g_i\|_{L^1}.
  \end{equation}

Define
 \[
 h=\sum_i\psi_i g_{i,\eps_i}, \quad b=2 { \sum_i}  \nabla\psi_i\cdot \nabla g_{i,\eps_i}.
 \]
Since $\sum_i (\nabla\psi_i\cdot \nabla g_i)=(\sum_i\nabla \psi_i) \cdot \nabla g_\delta=0$ {almost everywhere on $D$},
we have  \[
b=2\sum_i \nabla\psi_i\cdot (\nabla g_{i,\eps_i}-\nabla g_i).
\]

We compute
\[
\Delta h=\sum_i[(\Delta\psi_i)\, g_{i,\eps_i}+2\nabla\psi_i\nabla g_{i,\eps_i}+\psi_i\Delta g_{i,\eps_i}],
\]
and
since
\[
\sum_i(\Delta\psi_i ) g_i=\sum_i (\Delta\psi_i) g_\delta=0,
\]
we have
\[
\Delta h=\sum_i[\Delta\psi_i(g_{i,\eps_i}- g_i)+ 2\sum_i\nabla\psi_i\cdot(\nabla g_{i,\eps_i}-\nabla g_i)+\psi_i\Delta g_{i,\eps_i}].
\]
By~\eqref{oio-4}, ~\eqref{oio-6}, we may choose $\eps_i$ to be sufficiently small so that
\[
\|{(-\Delta-\lambda)} h\|_{L^1{(D)}}\leq 2 \|(-\Delta-\lambda) g_\delta\|_{L^1(D)}.
\]

\end{proof}


\end{document}